\documentclass[11pt,reqno]{amsart}
\usepackage[T1]{fontenc}
\usepackage[utf8]{inputenc}
\usepackage{lmodern}
\usepackage{amsmath,amssymb,mathtools,cite,xurl}
\usepackage{microtype}
\usepackage[a4paper,textwidth=15.7cm,textheight=23.5cm,centering]{geometry}
\usepackage{enumitem}
\setlist[enumerate]{label=\textup{(\roman*)},leftmargin=*,itemsep=3pt}
\usepackage{xcolor}
\usepackage{tikz}
\usepackage{float}
\usepackage{hyperref}
\definecolor{linkink}{RGB}{34,58,82}
\hypersetup{colorlinks=true,linkcolor=linkink,citecolor=linkink,urlcolor=linkink,
 pdftitle={\texorpdfstring{$2$}{2}-large sets are sets of nil-Bohr recurrence},
 pdfauthor={Ryan Alweiss and Sohail Farhangi},
 pdfsubject={An elementary special-generalized-polynomial obstruction to restricted arithmetic progressions}}
\allowdisplaybreaks[2]
\setlist[enumerate]{label=\textup{(\roman*)},leftmargin=*,itemsep=3pt}
\newtheorem{theorem}{Theorem}[section]
\newtheorem{proposition}[theorem]{Proposition}
\newtheorem{lemma}[theorem]{Lemma}
\newtheorem{corollary}[theorem]{Corollary}
\newtheorem{example}[theorem]{Example}
\newtheorem{conjecture}[theorem]{Conjecture}
\newtheorem*{bglconjecture}{Conjecture (Brown--Graham--Landman)}
\theoremstyle{definition}
\newtheorem{definition}[theorem]{Definition}
\newtheorem{question}[theorem]{Question}
\theoremstyle{remark}
\newtheorem{remark}[theorem]{Remark}
\numberwithin{equation}{section}
\newcommand{\N}{\mathbb N}
\newcommand{\Z}{\mathbb Z}
\newcommand{\R}{\mathbb R}

\newcommand{\eps}{\varepsilon}
\newcommand{\ud}{\overline d}
\newcommand{\Hpol}{\mathcal H}
\newcommand{\Kpol}{\mathcal K}
\newcommand{\nint}[1]{\langle #1\rangle}
\newcommand{\tnorm}[1]{\lVert #1\rVert_{\R/\Z}}
\newcommand{\coord}[1]{\lvert #1\rvert_{\infty}}
\newcommand{\AP}[3]{\{#1+j#2:0\leq j<#3\}}

\DeclareMathOperator{\UT}{UT}
\DeclareMathOperator{\osc}{osc}
\DeclareMathOperator{\rep}{rep}
\DeclareMathOperator{\lcm}{lcm}

\newcommand{\rotatedhash}{%
  \ooalign{%
    \hfil\rotatebox[origin=c]{-45}{$\scriptscriptstyle\Vert$}\hfil\cr
    \hfil\rotatebox[origin=c]{ 45}{$\scriptscriptstyle\Vert$}\hfil\cr
  }%
}
\newcommand{\dhash}{\mathrm{d}^{\rotatedhash}}
\newcommand{\cpart}[1]{\left\{#1\right\}_{\mathrm c}}
\title[$2$-large sets and nil-Bohr recurrence]{$2$-Large Sets and Nil-Bohr Recurrence: A Disproof of the $2$-Large Is Large Conjecture}
\author{Ryan Alweiss}
\thanks{R.A.: University of Cambridge and Trinity College, Cambridge, United Kingdom.
S.F.: Beijing Institute of Mathematical Sciences and Applications, Beijing, China.}
\address{Department of Pure Mathematics and Mathematical Statistics,
University of Cambridge, Centre for Mathematical Sciences,
Wilberforce Road, Cambridge CB3 0WB, United Kingdom;
and Trinity College, Cambridge CB2 1TQ, United Kingdom}
\email{ra699@cam.ac.uk}
\author{Sohail Farhangi}
\address{Beijing Institute of Mathematical Sciences and Applications,
No. 544, Hefangkou Village, Huaibei Town, Huairou District,
Beijing 101408, China}
\email{sohail.farhangi@gmail.com}
\date{}

\begin{document}
\begin{abstract}
A set $S$ of positive integers is $r$-large if every $r$-coloring of $\N$ contains arbitrarily long monochromatic arithmetic progressions with common difference in $S$.
We prove that every $2$-large set is a set of nil-Bohr recurrence, i.e., it meets every nil-Bohr neighborhood of $0$.
This proves a conjecture of the first author and gives a sharp answer to a question of Host, Kra, and Maass. 
Furthermore, we show that a piecewise syndetic set is $2$-large if and only if it is nil-Bohr recurrent.
From this, we show that a small modification of a set constructed by the first author is $2$-large but not $3$-large, thereby disproving the $2$-large is large conjecture of Brown, Graham, and Landman.
We also deduce that the family of $2$-large sets is not partition regular, thereby answering another question of Brown, Graham, and Landman.

We also show that, for every fixed nil-Bohr neighborhood $B$ of $0$ and every $\eps>0$, restricting the common difference to $B$ increases the extremal number for $k$-term arithmetic progressions by at most a factor of $1+\eps$, provided $k$ is sufficiently large.

\end{abstract}
\maketitle

\section{Introduction}
\label{sec:intro}

One of the oldest results in Ramsey theory is van der Waerden's theorem \cite{vdW}, which states that every finite coloring of $\N$ contains arbitrarily long monochromatic arithmetic progressions (APs).
In the 1990s, Brown, Graham, and Landman \cite{BGL} began the systematic study of the version of this theorem with restricted differences.

\begin{definition}
We say $S \subset \N$ is \emph{large} if every finite coloring of $\N$ has arbitrarily long monochromatic arithmetic progressions with common difference in $S$.
\end{definition}

Note van der Waerden's theorem is just the statement that $\N$ is large. Brown, Graham, and Landman give some examples of large sets $S$ in their paper, such as the image of a polynomial $P:\N\rightarrow\N$ with integer coefficients and $P(0)=0$, and also give several negative examples. They also defined a version of largeness with a fixed number of colors.

\begin{definition}
We say $S \subset \N$ is \emph{$r$-large} if every $r$-coloring of $\N$ has arbitrarily long monochromatic arithmetic progressions with common difference in $S$.
\end{definition}

Towards improving our understanding of which sets $S$ are large, Brown, Graham, and Landman posed several conjectures and questions, the most prominent of which is the following.

\begin{bglconjecture}[\cite{BGL}]
Every $2$-large set is large.
\end{bglconjecture}

It is more helpful to think about this conjecture in the contrapositive.
If there is a finite coloring so that there are no arbitrarily long monochromatic APs with common difference in $S$, then it is natural to ask if we can use only $2$ colors to block such APs.

The main reason for the difficulty of this conjecture is that until recently we had few examples of non-large sets.
The first natural class of non-large sets consists of complements of Bohr neighborhoods of $0$ and their subsets.
This class contains many of the sets considered by Brown, Graham, and Landman. Recall $||x||_{\R/\Z}=\min (\{x\}, 1-\{x\})$ is the distance from $x$ to the nearest integer.

\begin{definition}\label{def:Bohr}
    A set $B \subset \N$ is a \emph{Bohr neighborhood of $0$} if there exist $\alpha_1, \cdots, \alpha_k, \eps \in \R^{+}$ so that $B = \{n \in \mathbb{N}\ |\ ||n\alpha_i||_{\mathbb{R}/\mathbb{Z}} < \eps\ \forall\ 1 \le i \le k\}$.
    If $S \subseteq \mathbb{N}$ intersects every Bohr neighborhood of $0$, then $S$ is a \emph{set of Bohr recurrence}.
\end{definition}

Suppose that \(S\cap B=\emptyset\), where \(B\) is the Bohr neighborhood in Definition~\ref{def:Bohr}.
If we color $n$ by the image of $(n \alpha_1, \cdots, n\alpha_k) \in (\R/\Z)^k$ up to a small $\eps/10$ error in the $\ell^{\infty}$ norm, there are no two elements of the same color whose difference is in $S$, so certainly no long arithmetic progressions with difference in $S$.
We also mention here the celebrated \emph{Katznelson's Question}, which asks whether a set $S$ for which there exists a finite coloring of $\N$ with no two elements of the same color having a difference in $S$ must avoid a Bohr neighborhood of $0$.
See the introduction of \cite{AlweissObstacles} for more on Katznelson's Question.

The second author and Grytczuk \cite{FarhangiGrytczuk} first raised the question of whether a $2$-large set $S$ is necessarily a set of Bohr Recurrence.
This question was resolved in one dimension by Host, Kra, and Maass \cite[Lemma 7.4]{HKM}, and a quantitative variation of this result was recorded as \cite[Theorem 2]{FarhangiGrytczuk}.

\begin{theorem}[\cite{HKM}, \cite{FarhangiGrytczuk}]
    If $S$ is $2$-large, then there do not exist $\alpha, \eps \in \R^{+}$ so that $||s \alpha||_{\R/\Z}>\eps$ for all $s \in S$.
\end{theorem}

The proof is simple.
If this were the case, then one could just color $s$ red if $0 \le \{s \alpha \}<1/2$ and blue if $1/2 \le \{s \alpha \}<1$. In other words, one can pull back the two-coloring of $\R/\Z$ where the interval $[0,1/2)$ is red and $[1/2,1)$ is blue through the map $n \to \{\alpha n\}$.

Recently, the first author settled this question in full generality.

\begin{theorem}[\cite{AlweissBohr}]\label{thm:bohr}
    If $S$ is $2$-large, then $S$ is a set of Bohr recurrence.
\end{theorem}

The next natural class of non-large sets consists of those which are not sets of nil-Bohr recurrence.
For \(x\in\R\), let \(\nint{x}:=\lfloor x+1/2\rfloor\) denote its nearest integer, with ties resolved toward the larger integer, and let $\cpart{x}:=x-\nint{x}\in[-1/2,1/2)$ denote its centered fractional part.
Note that \(|\cpart{x}|=\tnorm{x}\).

Define recursively
\[
 \mathsf L(x_1)=x_1,\qquad
 \mathsf L(x_1,\ldots,x_s)=x_1\nint{\mathsf L(x_2,\ldots,x_s)}.
\]
Following \cite[Definition~4.1.2]{HSY}, we define a \emph{special generalized polynomial} to be an expression
\begin{equation}\label{eq:sgp-definition}
 P(n)=\mathsf L(\alpha_1n^{w_1},\ldots,\alpha_sn^{w_s}),
 \qquad \alpha_i\in\R,\quad w_i\in\N.
\end{equation}

We define the degree of $P$ by $\text{deg}(P)=w_1+\cdots+w_s$, and the \emph{depth} of $P$ by $d(P) = s$.
For example, $\alpha n$ has degree $1$ and $\beta n^2$ and
$\gamma n\nint{\theta n}$ have degree $2$.

\begin{definition}\label{def:NilBohr}
    The set $S \subset \N$ is a \emph{nil-Bohr neighborhood}\footnote{\label{fn:first}This definition is not the same definition used in \cite{HostKraNilBohrSetsOfIntegers,HSY}, but it results in the same notion of nil-Bohr recurrence as seen by Theorems B and 4.2.11 of \cite{HSY}.} \emph{of $0$} if there exist special generalized polynomials $P_1, \cdots, P_k$ and an $\eps > 0$ for which $S = \{n \in \mathbb{N}\ |\ ||P_i(n)||_{\R/\Z}<\eps\ \forall\ 1 \le i \le k\}$.
    If $S$ intersects every nil-Bohr neighborhood of $0$, then $S$ is a set of \emph{nil-Bohr recurrence}.
\end{definition}

In the language of ergodic theory, sets of nil-Bohr recurrence are sets of pointwise recurrence for nilrotations on nilmanifolds, whereas sets of Bohr recurrence are sets of pointwise recurrence for rotations on finite-dimensional tori. We will discuss this in more detail in Section \ref{sec:Nilmanifolds}.

Huang, Shao, and Ye \cite[Theorem A(1)]{HSY} showed that large sets are sets of nil-Bohr recurrence and asked whether or not the converse holds.
Recently, the first author \cite{AlweissObstacles} constructed an example of a set of nil-Bohr recurrence that is not large (in fact it is not $2$-large), thereby giving a negative answer to the question of Huang, Shao, and Ye.
Nonetheless, our main result shows that $2$-large sets are closely linked to sets of nil-Bohr recurrence.
To state it, we recall that a subset of $\mathbb{N}$ is \emph{piecewise syndetic} if it is the intersection of a set with bounded gaps (a \emph{syndetic} set) and a set containing arbitrarily long intervals (a \emph{thick} set).

\begin{theorem}\label{thm:2LargeIsNilrecurrent}
Let $S \subseteq \mathbb{N}$.
\begin{enumerate}
    \item If $S$ is $2$-large, then $S$ is a set of nil-Bohr recurrence.

    \item If $S$ is piecewise syndetic and a set of nil-Bohr recurrence, then $S$ is $2$-large. 
\end{enumerate}
\end{theorem}

We show in Corollary \ref{cor:2ButNot3Large} that Theorem \ref{thm:2LargeIsNilrecurrent}(ii) can be used to construct a $2$-large set $S$ that is not $3$-large, thereby disproving the $2$-large is large conjecture of Brown, Graham, and Landman.
Moreover, this set $S$ has the form $S = E\cup(\N\setminus3\N)$, where $E$ is a set of nil-Bohr recurrence that is not $2$-large.
Since $\N\setminus3\N$ also fails to be $2$-large, our example shows that the family of $2$-large sets is not partition regular, which answers another question of Brown, Graham, and Landman.
We will also show in Section \ref{sec:Nilmanifolds} that Theorem \ref{thm:2LargeIsNilrecurrent}(i) gives a sharp answer to Question 7.8 of Host, Kra, and Maass \cite{HKM}.

For $0<\delta<1$, we call $S\subseteq\N$ \emph{$\delta$-Szemer\'edi-large} if every $A\subseteq\N$ with natural upper density $\ud(A) \ge \delta$ contains arbitrarily long arithmetic progressions with
common difference in $S$, and $S$ is \emph{Szemer\'edi-large} if it is $\delta$-Szemer\'edi-large for every $\delta \in (0,1)$.
For a fixed $k$, we use the terms \emph{$(r,k)$-large} and \emph{$(\delta,k)$-Szemer\'edi-large} when only $k$-term progressions are required in the corresponding definitions.
It is immediate that $(\delta,k)$-Szemer\'edi-largeness implies $(\lfloor\delta^{-1}\rfloor,k)$-largeness.
We do not know if for every $r \ge 2$, there exists a $\delta(r) \in (0,1)$, independent of $k$, such that $(r,k)$-largeness implies $(\delta(r),k)$-Szemer\'edi-largeness for every $k \ge 2$.
Nonetheless, we prove the following result.

\begin{theorem}\label{thm:deltaSzLargeIsNilrecurrent}
    If $S$ is $\delta$-Szemer\'edi-large for some $\delta \in (0,1)$, then $S$ is a set of nil-Bohr recurrence.
\end{theorem}

In fact, we prove the following \textit{polychromatic} result from which we can quickly deduce Theorems \ref{thm:2LargeIsNilrecurrent}(i) and \ref{thm:deltaSzLargeIsNilrecurrent}.

\begin{theorem}\label{thm:omitted}
For every $r\geq2$ and $m,D\geq1$, there are constants $C=C(r,m,D)>0$ and $k_0=k_0(r,m,D)\geq2$ with the following property.
For every family of special generalized polynomials $P_1,\ldots,P_m$ of degrees $1\leq d_i\leq D$, there is an $r$-coloring $c$ of $\N$ such that, for all $a,d\in\N$ and all $k\geq k_0$, if the progression $\{a, a+d, \cdots, a+(k-1)d\}$ does not contain elements of all $r$ colors, then $||P_i(d)||_{\R/\Z} \le Ck^{-d_i}$ for $1 \le i \le m$.

The coloring is fixed independently of $k$, and the constants are independent of all the real coefficients.
\end{theorem}

We see that Theorem~\ref{thm:2LargeIsNilrecurrent}(i) follows directly from Theorem~\ref{thm:omitted} by taking $r=2$.
For Theorem~\ref{thm:deltaSzLargeIsNilrecurrent}, take the $P_i$ witnessing the failure of nil-Bohr recurrence of $S$ and $r=\lceil(1-\delta)^{-1}\rceil$ in Theorem~\ref{thm:omitted}, and let $C_i$ be the elements of color $i$ for $1 \le i \le r$.
Some \(\N\setminus C_i\) has upper density at least \(1-1/r\geq\delta\). By Theorem~\ref{thm:omitted} for sufficiently large $k$, any $k$-AP with common difference in $S$ contains elements of all $r$ colors and thus is not contained in $\N \setminus C_i$. Hence $S$ is not $\delta$-Szemer\'edi-large.

The same argument applies to any density
\(\mathfrak d:\mathcal P(\N)\to[0,1]\) so that for every partition $\mathbb{N} = \bigcup_{i = 1}^rC_i$, there exists $1 \le i \le r$ for which $\mathfrak{d}(\mathbb{N}\setminus C_i) \ge 1-r^{-1}$.
This property, rather than finite subadditivity alone\footnote{Define \(\nu(A)\) to be \(0\) if \(A\) is finite,
\(1\) if \(A\) is cofinite, and \(1/2\) otherwise.
Then \(\nu\) is normalized, monotone, translation-invariant, and
finitely subadditive.
Nonetheless, for every $r \ge 3$ the partition
\(\N=\bigcup_{i = 1}^rC_i\) into residue classes modulo \(r\) satisfies \(\nu(\N\setminus C_i)=1/2<1-r^{-1}\) for all $1 \le i \le r$.}, is what the argument uses.
Hence it holds for upper densities along fixed additive or multiplicative F{\o}lner sequences, additive or multiplicative
upper Banach density, and the density \(\dhash\) defined in \cite{RichterDensity}.
Thus the corresponding versions of Theorem~\ref{thm:deltaSzLargeIsNilrecurrent} hold when \(\delta\)-Szemer\'edi-largeness is defined using \(\mathfrak d(A)\geq\delta\).

\begin{remark}\label{rem:PolychromaticRamseyTheory}
    The passage from Theorem~\ref{thm:omitted} to its partition and density consequences uses a standard principle from the theory of polychromatic colorings.
Such a coloring partitions the vertices of a hypergraph into transversals, each meeting every edge; see \cite[Section~1.1]{BollobasEtAl2013}. Consequently, merging the
colors into two nonempty groups gives a coloring with no monochromatic edge, while the complement of each original color class is independent.
In the finite setting, deleting a smallest color class retains at least a proportion \(1-1/r\) of the vertices;
this observation is recorded explicitly in \cite[Observation~4]{Palvolgyi2024}.
A corresponding density argument for translates of a fixed finite subset of \(\Z\) appears in \cite[Lemma~1]{AxenovichEtAl2019}, where polychromatic colorings yield blocking sets of small density.
Our upper-density deduction above applies the same complement-of-one-color principle to the family of progressions whose differences violate the prescribed approximation conditions.
\end{remark}

For a finite set $\Omega\subseteq\N$ and $B\subseteq\N$, define
\begin{align}
 r_k(\Omega)&=\max\{|A|:A\subseteq\Omega\text{ contains no $k$-term AP}\},
 \label{eq:rk}\\
 r_k(\Omega;B)&=\max\{|A|:A\subseteq\Omega\text{ contains no $k$-term AP
 with common difference in }B\}.
 \label{eq:rkB}
\end{align}
Determining the growth rate of the Szemer\'edi numbers $r_k(N) := r_k([1,N])$ is an important problem in additive combinatorics (see e.g. \cite{ImprovedBoundsForSzemeredi,OAI:Quasipolynomial-Bounds-for-Arithmetic-Progressions-September-23-2026}).
Interestingly, Theorem \ref{thm:omitted} implies that restricting the difference to be in a nil-Bohr set does not change these numbers much if $k$ is sufficiently large compared to the parameters of the nil-Bohr set.
This is more evidence of the tight connection between higher-order Fourier analysis and the behavior of $r_k(N)$.

\begin{corollary}\label{cor:extremal}
Fix $\eps,c>0$ and $m,D\geq1$. There is
$k_0=k_0(\eps,c,m,D)$ such that, for every family $P_1,\ldots,P_m$ of
special generalized polynomials of degrees $1\leq d_i\leq D$, every
$k\geq k_0$, and every finite $\Omega\subseteq\N$, the set
\begin{equation}\label{eq:Bk}
 B_k(c)=\{d\in\N:\tnorm{P_i(d)}\leq ck^{-d_i}\text{ for }1\leq i\leq m\}
\end{equation}
satisfies
\begin{equation}\label{eq:extremal}
 r_k(\Omega)\leq r_k(\Omega;B_k(c))
 \leq(1+\eps)r_k(\Omega).
\end{equation}
\end{corollary}

We emphasize that the threshold $k_0(\eps,c,m,D)$ does not depend on the coefficients of $P_i$, and also that $B_{k+1}(c) \subseteq B_k(c)$.
\par\vspace{\baselineskip}



\noindent\textbf{Organization:}
In Section \ref{sec:BohrCase} we give a new proof of Theorem \ref{thm:bohr} based on the proof of Theorem \ref{thm:omitted}.
In Section \ref{ssec:Outline} we give an outline of how our proof of Theorem \ref{thm:bohr} generalizes to a proof of Theorem \ref{thm:omitted}.
In Sections \ref{ssec:PolynomialMatricesAndIntegerCorrections}-\ref{ssec:recovery} we gather all of the technical lemmas needed to give a fully rigorous proof of Theorem \ref{thm:omitted} in Section \ref{ssec:TheProof}.
In Section \ref{sec:Nilmanifolds} we introduce nilmanifolds and nilrotations, explain why Theorem~\ref{thm:2LargeIsNilrecurrent}(i) gives a sharp answer to \cite[Question 7.8]{HKM}, and prove Theorem~\ref{thm:2LargeIsNilrecurrent}(ii).
In Section \ref{sec:extremal} we prove Corollary \ref{cor:extremal}.
Lastly, in Section \ref{sec:questions} we consider some directions for future research.
\par\vspace{\baselineskip}

\noindent\textbf{Acknowledgements:} The second author would like to thank Bruce Landman and Abdollah Khodkar for hosting the 2014 research experience for undergraduates (REU) at the University of West Georgia.
It was at this REU that he was introduced to the 2-large is large conjecture, and began his explorations in Ramsey theory.
We would also like to thank Yuval Peres for suggesting that we try to find a way to salvage the 2-large is large conjecture, which has resulted in Question \ref{q:PartitionRegular2LargeIsLarge}.
\par\vspace{\baselineskip}

\noindent\textbf{AI Statement:} This paper is the product of roughly 15 months of work with some breaks in the middle. 
The initial plan of attack to prove the first half of Theorem \ref{thm:2LargeIsNilrecurrent} was to extend the methods of \cite{AlweissBohr} involving random colorings to the setting of nilrotations.
While most of the ideas extended easily, there were new technical difficulties that arose from working with nilmanifolds.
Initially, AI was used to tutor the authors about nilmanifolds, Mal'cev bases, equidistribution, Lie groups, and differential geometry so that they could try to rigorously prove the first half of Theorem \ref{thm:2LargeIsNilrecurrent}.
When this paper was largely completed, ChatGPT-6 was released, so the authors uploaded their draft into GPT-6 to have it complete the final details involving technical calculations with Mal'cev bases and equidistribution.
Then, the authors asked GPT-6 if there existed a simpler proof of the first half of Theorem \ref{thm:2LargeIsNilrecurrent} that avoids the heavy machinery of nilmanifolds and only uses elementary calculations involving generalized polynomials, which resulted in the shell coloring proof of the present paper.
The authors then discussed with GPT-6 the possibility of also proving Theorem \ref{thm:deltaSzLargeIsNilrecurrent}, which led them to Theorem \ref{thm:omitted}.
The statement and proof of Corollary \ref{cor:extremal} were given to the authors by GPT-6 after it was asked for applications of the main results.
After version 2 of this work was uploaded to the arXiv, the authors prompted GPT-6 to solve the $2$-large is large conjecture, and it did so by providing a statement and proof of the second half of Theorem \ref{thm:2LargeIsNilrecurrent}.

The website \url{https://sohailfarhangi.com/2-large-sets-are-sets-of-nil-bohr-recurrence/} contains recreations of many of the ChatGPT chat logs of the second author pertaining to this paper, samples of some of the pdfs that were produced by ChatGPT in the process, and some earlier drafts of this paper based on the random coloring method instead of the shell coloring method of the present paper.
The YouTube video \url{https://www.youtube.com/watch?v=aMKdZebRVI4&t=83s} shows the second author going through many of the ChatGPT chat logs pertaining to this project and discussing them.

\section{The shell coloring in the Bohr case}\label{sec:BohrCase}

We first present an alternate proof of Theorem~\ref{thm:bohr}, which we will later generalize to prove Theorem~\ref{thm:omitted}.
This proof uses a \textit{shell coloring} rather than the random coloring of \cite{AlweissBohr}. As discussed above, the random-coloring approach of \cite{AlweissBohr} can also be extended to nilrotations. A common feature of torus rotations and nilrotations is their low orbit complexity, reflected in their zero topological entropy, which underlies the probabilistic argument. Here we use a shell coloring to obtain a simpler proof.

\begin{proof}[Proof of Theorem \ref{thm:bohr}]
Let $\alpha_1,\ldots,\alpha_d,\eps\in\R^{+}$
and let $S$ be the set of $x$ such that
$\displaystyle\max_{1\le i\le d}\|\alpha_i x\|_{\R/\Z}>\eps$.
We will pick a sufficiently large integer $K$ as a function of $d$ and $\eps$, and then pick some $L$ divisible by all integers up to $M^d$ for some sufficiently large function $M$ of $K$. In particular, $M=K^2$ and $L=(M^d)!$ suffice.
As in \cite{AlweissBohr}, we define a coloring of $(\R/\Z)^d$ (actually, in \cite{AlweissBohr} we color $(\R/L\Z)^d$, but this normalization will be more convenient for the general case), then pull it back to $\Z$ via the map $h=g \circ f: \Z \to \R^d \to (\R/\Z)^d$, where the first map $f$ sends
$n\mapsto(\alpha_1n/L,\ldots,\alpha_dn/L)$
and the second map $g$ is the natural quotient map by $\Z^d$, with fundamental domain $[-1/2,1/2)^d$.

However, rather than splitting this into $M^d$ randomly colored cells as in
\cite{AlweissBohr}, we color the point
$(x_1,\ldots,x_d)\in[-1/2,1/2)^d$ by the parity of $\left\lfloor  K\sum_{i=1}^{d}x_i^2\right\rfloor$, so that if this floor is even then the point is red and if it is odd then the point is blue.

For $y \in S$, consider the image of $x, x+y, \cdots, x+ky$ under the map $h$. We will prove that when $k$ is sufficiently large and divisible by a sufficiently large factorial, this image contains points of both colors. We partition $(\R/\Z)^d$ into $M^d$ boxes of side length $1/M$ in the obvious way.
By pigeonhole, there is some $0 \le a<b \le M^d \ll k$ so that $h(x+ay)$ and $h(x+by)$ are in the same box. Hence $h(x+ay)$ and $h(x+by)$ differ by at most $1/M$ in each coordinate, and on $(\R/\Z)^d$ the same is true of $h(x)$ and $h(x+cy)$ where $1 \le c=b-a \le M^d$. Let $v\in[-1/M,1/M]^d$ be a representative of $h(cy)$.

Now, some coordinate $y'$ of $f(y)$ satisfies
$\|Ly'\|_{\R/\Z}>\eps$, and the same coordinate of $f(cy)$ is $cy'$.
Write $cy'=z+\eta$, where $z$ is an integer and $\eta$ is the corresponding coordinate of $v$.
Then $Ly'=z\frac{L}{c}+\frac{L\eta}{c}$. Since $c \mid L$, this means that
$\|\frac{L\eta}{c}\|_{\R/\Z}=\|Ly'\|_{\R/\Z}>\eps$,
so we must have $|\frac{L\eta}{c}|>\eps$ and hence
$|\eta|>c\eps/L$.

Since $c \mid k$, let $J=k/c$, and consider the linear path on the torus
\[
\{(f(x)+tv)+\Z^d:0\leq t\leq J\}\subseteq(\R/\Z)^d.
\]

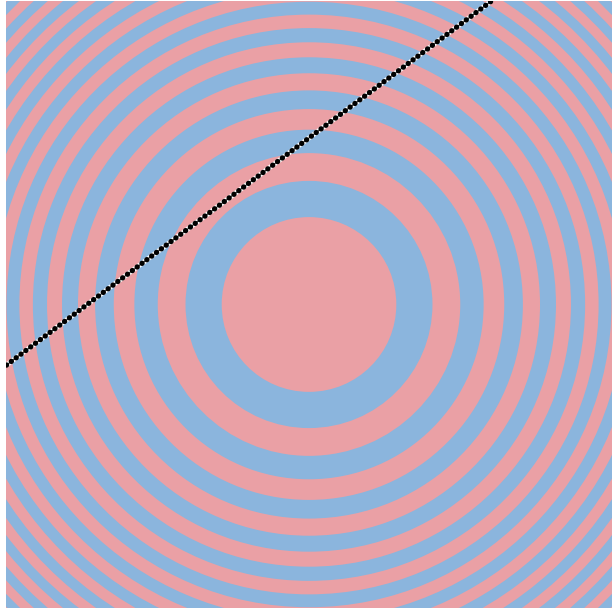
\begin{figure}[H]
\centering
\begin{tikzpicture}[x=8cm,y=8cm]
  \definecolor{shellred}{HTML}{EAA0A5}
  \definecolor{shellblue}{HTML}{8BB5DD}
  \def\ShellK{48}

  \path[use as bounding box] (-.5,-.5) rectangle (.5,.5);
  \pgfmathtruncatemacro{\LastShell}{floor(\ShellK/2)}

  \begin{scope}
    \clip (-.5,-.5) rectangle (.5,.5);

    \foreach \j in {\LastShell,...,0}{
      \pgfmathsetmacro{\Radius}{sqrt((\j+1)/\ShellK)}
      \pgfmathtruncatemacro{\Parity}{mod(\j,2)}
      \ifnum\Parity=0
        \fill[shellred] (0,0) circle[radius=\Radius];
      \else
        \fill[shellblue] (0,0) circle[radius=\Radius];
      \fi
    }

    \draw[white,line width=1.8pt]
      (-.5,-.1) -- (.3,.5);
    \draw[black!85,line width=.4pt]
      (-.5,-.1) -- (.3,.5);

    \foreach \k in {0,...,100}{
      \pgfmathsetmacro{\PointX}{-.5+.8*\k/100}
      \pgfmathsetmacro{\PointY}{-.1+.6*\k/100}
      \fill[black] (\PointX,\PointY) circle[radius=1pt];
    }
  \end{scope}
\end{tikzpicture}
\caption{The coloring of $[-1/2,1/2]^2$ with $K=48$, alternately colored shells, and marked points equally spaced along a segment crossing the square.}
\label{fig:shell-progression}
\end{figure}

This path will contain $h(x), h(x+cy), h(x+2cy), \cdots, h(x+ky)$. In particular, since $k$ is chosen large enough and some coordinate of $v$ has magnitude greater than $c\eps/L$, this path has length greater than $k\eps/L\gg1$.
Furthermore, $h(x), h(x+cy), h(x+2cy), \cdots, h(x+ky)$ will be $\frac{d}{M}$-dense in this path in the $\ell^1$ norm.
This path must completely cross at least $\Omega_d(K)$ of the shells. The $\frac{d}{M}$-density condition comfortably guarantees that $\{h(x), h(x+cy), \cdots, h(x+ky)\}$ contains a point from each completely crossed shell, and so $\{h(x), h(x+cy), \cdots, h(x+ky)\}$ contains points of both colors.
Hence, in our coloring of $\Z$ obtained by pulling back the shell coloring along $h$, the sub-progression $\{x, x+cy, \cdots, x+ky\}$ of $\{x, x+y, \cdots, x+ky\}$ contains points of both colors as well.
\end{proof}

We remark that a careful analysis of the bound here shows that
we can take
\[
k \le \frac{\exp((Cd^2)^d)}{\eps},
\]
double exponential in $d$ and linear in $1/\eps$.
For the coloring in \cite{AlweissBohr}, a careful analysis gives
\[
k \le \frac{\exp((Cd\log(2d))^{d})}{\eps},
\]
better but also double exponential in $d$ and linear in $1/\eps$.

\begin{question}
Given $d$ and $\eps$, what is the smallest $k$ such that,
for every $\alpha_1,\ldots,\alpha_d$, there is a two-coloring
of $\N$ in which every arithmetic progression
$x,x+y,\ldots,x+ky$ satisfying
\[
\max_{1\le i\le d}\|\alpha_i y\|_{\R/\Z}>\eps
\]
contains points of both colors?
\end{question}

Any progress on this question would be welcome.
Can the dependence on $d$ be improved from double exponential
to single exponential? Progress on quantitative bounds in
the more general nil-Bohr case would also be welcome. In particular, it would be interesting to improve the bounds on $k_0$ in weaker versions of Theorem \ref{thm:omitted} and Corollary \ref{cor:extremal}, where the defining radii of the nil-Bohr sets decrease more
slowly with $k$, or even remain fixed independently of $k$.
For example, one could require only
$\|P_i(d)\|_{\R/\Z}\le\eps$ for every $i$, for a fixed
$\eps>0$, with $k_0$ allowed to depend on $\eps$.

\section{Proof of Theorem \ref{thm:omitted}}

We will now prove our main result, Theorem \ref{thm:omitted}. We will first describe the coloring and give an outline of the general argument.

\subsection{Outline of the argument}\label{ssec:Outline}

In the general case, we will again use a shell coloring, but rather than working with line segments on tori, which are products of circle groups, we must consider polynomial curves and products of upper-unitriangular matrix groups.

For each special generalized polynomial \(P_i\) with $1 \le i \le m$, we will construct an associated matrix \(g_{P_i}(t)\).
In particular, if $P_i=\mathsf L(\alpha_{i,1}t^{w_{i,1}}, \cdots, \alpha_{i,s_i}t^{w_{i,s_i}})$ then

\[
g_{P_i}(t)=
\begin{pmatrix}
1 & -\alpha_{i,1}t^{w_{i,1}} & 0 & \cdots & 0 & 0\\[6pt]
0 & 1 & -\alpha_{i,2}t^{w_{i,2}} & \cdots & 0 & 0\\[6pt]
\vdots & \ddots & \ddots & \ddots & \vdots & \vdots\\[6pt]
0 & \cdots & 0 & 1 & -\alpha_{i,s_i-1}t^{w_{i,s_i-1}} & 0\\[6pt]
0 & \cdots & 0 & 0 & 1 & \alpha_{i,s_i}t^{w_{i,s_i}}\\[6pt]
0 & \cdots & 0 & 0 & 0 & 1
\end{pmatrix}
\ \in\ \mathrm{UT}_{s_i+1}(\R).
\]

Let \(g(t)=(g_{P_1}(t), \cdots, g_{P_m}(t))\) be the $m$-tuple of the associated matrices, let \(G=\prod_{i=1}^{m} \mathrm{UT}_{s_i+1}(\R)\) be the corresponding product of matrix groups containing the $g_{P_i}(t)$, and let \(\Gamma=\prod_{i=1}^{m} \mathrm{UT}_{s_i+1}(\Z)\) be its subgroup of integer matrices.
For the block associated to \(P_i\), assign the matrix coordinate \((u,v)\), with \(u<v\), the weight
\[
w_{i;uv}:=w_{i,u}+w_{i,u+1}+\cdots+w_{i,v-1}.
\]
These matrices encode generalized polynomials.
In particular, for each $t \in \mathbb{R}$, there is some $\gamma_t \in \Gamma$ such that each $\cpart{P_i(t)}$ appears as an entry of the matrix $g(t)\gamma_t$, which is also the canonical representative of $g(t)$ in $G/\Gamma$ (cf. Example \ref{ex:Reduction}).

Our coloring will generalize the sum of squares of the $\R/\Z$ norm from the Bohr case. In particular, we will use the continuous right-\(\Gamma\)-invariant function

\[
F(A)= \min_{\gamma\in\Gamma}
      \sum_{i<j}(A\gamma)_{ij}^{\,2}=
\min_{(\gamma_1,\ldots,\gamma_m)\in\Gamma}
\sum_{i=1}^{m}
\sum_{1\le u<v\le s_i+1}
(A_i\gamma_i)_{uv}^{2}.
\]

In other words, for each of the upper triangular matrices in $\mathrm{UT}_{s_i+1}(\R)$, we multiply it by some element of $\mathrm{UT}_{s_i+1}(\Z)$ to minimize the sum of squares of its off-diagonal entries, and after doing this for all matrices we take the sum of squares of the off-diagonal entries.


A nonconstant ordinary polynomial matrix curve cannot remain in one level set of \(F\).
Indeed, on a compact parameter interval, \(F\) is the minimum of finitely many nonconstant sums of squares of polynomials.

Working on a suitable compact family $\mathcal{K}$ of polynomial curves gives a uniform positive lower bound for the oscillation of polynomial curves of bounded degree making a fixed-size excursion inside a bounded coordinate box.
We color successive thin shells cyclically with \(r\) colors.
More concretely, let $\kappa:G\rightarrow\{0,1,\cdots,r-1\}$ be given by $\kappa(A) = \left\lfloor\frac{F(A)}{h}\right\rfloor\pmod{r}$ for a sufficiently small $h$.
This produces a coloring under
which every sufficiently fine sample of such an excursion meets all \(r\) colors.
We then color the integers by \(c(n)=\kappa(g(n/L))\), for a fixed integer \(L\) chosen below.

Suppose that a sufficiently long \(k\)-term progression \(\{a,a+d,\ldots,a+(k-1)d\}\) omits a color, and put \(p(t)=g((a+dt)/L)\).
Let $\nu = \sum_{i = 1}^m\binom{s_i+1}{2}$ and view $[-\frac{1}{2},\frac{1}{2})^\nu$ as a fundamental domain for $G/\Gamma$.
An auxiliary finite coloring of $[-\frac{1}{2},\frac{1}{2})^\nu$ records the small box containing the canonical representative of each \(p(j)\).
Van der Waerden's Theorem provides an $M$-term subprogression, with bounded starting index \(b\) and step \(q\), whose representatives lie in one box.
Lemma \ref{lem:rounding} then gives
\begin{equation}\label{eq:Outline}
     p(b+qt)C(t)=Q(t),\qquad C(n)\in\Gamma\quad(n\in\Z)
\end{equation}

where \(C\) and \(Q\) have ordinary polynomial entries with degree bounds $\text{deg}(C_{i;uv}),\text{deg}(Q_{i;uv}) \le w_{i;uv}$, and $Q(n)$ is the canonical representative of $p(b+qn)$ for all $0 \le n < M$.
Since $\kappa(p(b+qn)) = \kappa(Q(n))$ for $n \in \mathbb{Z}$, we see that $\{\kappa(Q(n))\}_{n = 0}^J$ omits at least one of the $r$ colors, where $J := \left\lfloor\frac{k-1-b}{q}\right\rfloor$.

A first-exit argument now keeps \(Q\) bounded on a real interval \([0,J]\).
Suppose instead that \(Q\) first escapes a distance \(\rho\) from its starting point at some time \(T\leq J\).
Rescaling the parameter gives a curve
\(\widetilde Q(u):=Q(Tu)\) in the compact set \(\mathcal K\).
This polynomial curve plays the role of the linear path
on the torus in the Bohr case.
The interpolation estimate in Lemma \ref{lem:interpolation} ensures that \(T\) is large enough for the all-color sampling lemma to apply.
Since \(\widetilde Q(j/T)=Q(j)\), the resulting fine sample of the normalized curve consists of integer samples of the original lift, and these must receive all \(r\) colors, yielding the desired contradiction.
Thus \(Q\) remains bounded on $[0,J]$, and polynomial interpolation on \([0,J]\) bounds the coefficient of degree \(w\) in each weight-\(w\) coordinate by \(O(J^{-w})\).

The leading coefficient-extraction map \(\tau\) removes the
starting point of the progression. The choice of \(L\) ensures
that applying the dilation map \(\delta_{L/q}\) both removes the
auxiliary step $q$ and clears the factorial denominators in
\(\tau(C)\) (cf. Section \ref{ssec:FiniteDifferences}).
In particular, applying \(\delta_{L/q}\circ\tau\) to Equation~\eqref{eq:Outline} yields
\[
g(d)\gamma_d=E,\qquad \gamma_d\in\Gamma,\qquad |E_{i;uv}|\ll_{r,m,D}k^{-w_{i;uv}}.
\]
For sufficiently large \(k\), we see that $E$ is the canonical representative of $g(d)$, so the values \(\cpart{P_i(d)}\) appear as entries of $E$, and we obtain the required bound
\(\tnorm{P_i(d)}\ll_{r,m,D}k^{-d_i}\).
All parameters defining the shell coloring and \(L\) are chosen independently of the real coefficients of the $P_i$ and the progression length.

\subsection{Polynomial matrices and integer corrections}\label{ssec:PolynomialMatricesAndIntegerCorrections}

Fix the exponent list of each of the $m$ expressions for the moment.
For the $s$th list $(w_{s,1},\ldots,w_{s,\ell_s})$, let
$G_s=\UT_{\ell_s+1}(\R)$, the group of upper-triangular matrices with
ones on the diagonal, and set
\begin{equation}\label{eq:product-group}
 G=\prod_{s=1}^m\UT_{\ell_s+1}(\R),\qquad
 \Gamma=\prod_{s=1}^m\UT_{\ell_s+1}(\Z),\qquad
 \nu=\sum_{s=1}^m\binom{\ell_s+1}{2}.
\end{equation}
Thus $\nu$ is the total number of off-diagonal coordinates, not the
number of expressions. All coordinate formulas below are understood
blockwise, and sums and maxima include all the blocks.
In one block assign position $(i,j)$, $i<j$, the weight
\begin{equation}\label{eq:weights}
 w_{ij}=w_i+\cdots+w_{j-1}.
\end{equation}
Every weight is positive and at most $D$, and
$w_{ik}+w_{kj}=w_{ij}$.
Write $\coord{A-B}=\max_{i<j}|A_{ij}-B_{ij}|$ for the coordinate distance.

Let $\Hpol$ be the class of polynomial matrix maps $A:\R\to G$ with
$\deg A_{ij}\leq w_{ij}$. We define two maps $\delta_u:G\to G$ and
$\tau:\mathcal H\to G$: the first dilates each coordinate
according to its weight, while the second extracts the
coefficient of the corresponding weighted degree.
For $B\in G$, $A\in\mathcal H$, and $i<j$, set

\begin{equation}\label{eq:dilation-tau}
 (\delta_u B)_{ij}=u^{w_{ij}}B_{ij}\quad(u>0),\qquad
 \tau(A)_{ij}=[t^{w_{ij}}]A_{ij}(t).
\end{equation}

We will only consider $\delta_u$ and $\tau$ on matrices and matrix maps whose diagonals are all $1$s, and clearly they preserve this property. The symbol $[t^w]$ denotes coefficient extraction.

\begin{lemma}\label{lem:algebra}
The class $\Hpol$ is closed under multiplication, inversion, and affine
changes of parameter. The map $\delta_u$ is an automorphism of $G$, and
\begin{equation}\label{eq:tau-algebra}
 \tau(AB)=\tau(A)\tau(B),\qquad
 \tau\bigl(A(b+qt)\bigr)=\delta_q\tau(A)\quad(q>0).
\end{equation}
Moreover, $\tau(A)=I$ if $A$ is constant.
\end{lemma}
\begin{proof}
For upper-unitriangular matrices,
\begin{equation}\label{eq:matrix-product}
 (AB)_{ij}=A_{ij}+B_{ij}+\sum_{i<k<j}A_{ik}B_{kj}.
\end{equation}
Additivity of the weights bounds the degree of every term by $w_{ij}$.
The inverse of $I+V$, with $V$ strictly upper triangular, is the finite
sum $I-V+V^2-\cdots$; every path of entries contributing to position
$(i,j)$ again has total weight $w_{ij}$. This proves the closure
assertions. The same additive identity proves that $\delta_u$ respects
multiplication; its inverse is $\delta_{1/u}$.

In a product $A_{ik}B_{kj}$, a term of degree
$w_{ij}=w_{ik}+w_{kj}$ can arise only from the terms of the respective
maximal allowed degrees. Taking this coefficient in
\eqref{eq:matrix-product} proves the first identity in
\eqref{eq:tau-algebra}. The second follows from the scalar identity
\[
 [t^w]f(b+qt)=q^w[t^w]f(t)\qquad(\deg f\leq w).
\]
Finally, constant entries contribute zero because their weights are
positive.
\end{proof}

For example, for adjacent weights $2,1$, the top-right weight is $3$.
In a product, its degree-three coefficient retains the cross term
$([t^2]A_{12})([t]B_{23})$. Thus $\tau$ is a matrix homomorphism, not
an approximation that discards the noncommutative terms. The scalar
identity
\begin{equation}\label{eq:scalar-extraction}
 [t^w]\bigl(\alpha(a+dt)^w\bigr)=\alpha d^w
\end{equation}
explains why this operation will eliminate the starting point $a$.

\subsection{Canonical representatives}
Put $I=[-1/2,1/2)$.
Every right coset $A\Gamma$ has a unique representative $\rep(A)$ whose $\nu$ off-diagonal entries lie in $I$.
To construct it, seek $\lambda\in\Gamma$ with $A\lambda\in I^\nu$
in these coordinates and process entries in increasing order of the gap
$j-i$. Once the shorter-gap entries are known, put
\begin{equation}\label{eq:canonical-reduction}
 f_{ij}=A_{ij}+\sum_{i<k<j}A_{ik}\lambda_{kj},\qquad
 \lambda_{ij}=-\nint{f_{ij}}.
\end{equation}
Then $(A\lambda)_{ij}\in I$, and this is the unique integer choice
with that property. Induction proves existence and uniqueness, and also
\begin{equation}\label{eq:rep-invariance}
 \rep(A\gamma)=\rep(A)\qquad(\gamma\in\Gamma).
\end{equation}
This representative need not depend continuously on $A$.

\subsection{Finite differences and integer-valued polynomials}\label{ssec:FiniteDifferences}
Write $\Delta_s f(n)=f(n+s)-f(n)$ and $\Delta = \Delta_1$.
We use
\begin{equation}\label{eq:diff-formula}
 \Delta_s^j f(n)=\sum_{v=0}^j(-1)^{j-v}\binom jv f(n+sv).
\end{equation}
A polynomial of degree at most $w$ has vanishing $(w+1)$st difference.
Conversely, a finite sequence of length at least $w+1$ whose
$(w+1)$st differences vanish agrees with a polynomial of degree at most
$w$: interpolate its first $w+1$ values, and use the difference equation
to determine each remaining value. Its Newton interpolation formula is
\begin{equation}\label{eq:newton}
 f(t)=\sum_{v=0}^w\Delta^v f(0)\binom tv.
\end{equation}
Consequently, a polynomial $f$ of degree at most $w$ that is
integer-valued at $0,\ldots,w$ is integer-valued on all of $\Z$, and
\begin{equation}\label{eq:integer-leading}
 w!\,[t^w]f(t)=\Delta^w f(0)\in\Z.
\end{equation}
Here $\binom nv$ is an integer even when $n<0$, by
$\binom{-u}{v}=(-1)^v\binom{u+v-1}{v}$ for $u\geq1$.
The polynomial $\binom t2=(t^2-t)/2$ shows why integer-valuedness does
not mean that the monomial coefficients are integers.

We need another lemma to deal with boundary effects from quotienting to the fundamental domain, so that we can reduce the problem to one where we have ordinary polynomials rather than generalized polynomials.

\begin{lemma}[Polynomial rounding]\label{lem:rounding}
Let $M\geq D+2$, and partition $I^\nu$ into boxes whose coordinate
intervals have length at most $\xi$, where
$2^{D+1}\xi<1$. If $A\in\Hpol$ and
$\rep(A(0)),\ldots,\rep(A(M-1))$ lie in one box, then there are
$Q,C\in\Hpol$ such that
\begin{equation}\label{eq:rounding}
 A(t)C(t)=Q(t)\quad(t\in\R),\qquad
 C(n)\in\Gamma\quad(n\in\Z),
\end{equation}
with $Q(n)=\rep(A(n))$ for $0\leq n<M$. Furthermore,
\begin{equation}\label{eq:rounding-factorial}
 w_{ij}!\,\tau(C)_{ij}\in\Z
\end{equation}
for every off-diagonal position.
\end{lemma}
\begin{proof}
On the initial samples write
$A(n)\lambda(n)=r(n)=\rep(A(n))$, with $\lambda(n)\in\Gamma$.
We construct $C_{ij}$ and $Q_{ij}$ in increasing order of the gap,
maintaining $C_{ij}(n)=\lambda(n)_{ij}$ and $Q_{ij}(n)=r_{ij}(n)$
for $0\leq n<M$.

For the current position $(i,j)$ put $w=w_{ij}$ and
\[
 f_{ij}(t)=A_{ij}(t)+\sum_{i<k<j}A_{ik}(t)C_{kj}(t).
\]
The entries in the sum have already been constructed.
Their weighted degree bounds give $\deg f_{ij}\leq w$, while the induction hypothesis
and \eqref{eq:matrix-product} give
\[
 r_{ij}(n)-f_{ij}(n)=\lambda(n)_{ij}\in\Z
 \qquad(0\leq n<M).
\]
Thus $\Delta^{w+1}r_{ij}(n)$ is an integer wherever it is defined.
All the $r_{ij}(n)$ lie in an interval of length at most $\xi$.
Subtracting an endpoint of that interval does not change a positive-order
difference. Hence \eqref{eq:diff-formula} yields
\[
 |\Delta^{w+1}r_{ij}(n)|\leq2^{w+1}\xi<1.
\]
These integer differences therefore vanish exactly.

The finite interpolation observation gives a polynomial $Q_{ij}$ of
degree at most $w$ agreeing with all the $r_{ij}(n)$. Set
$C_{ij}=Q_{ij}-f_{ij}$. It has degree at most $w$ and is integer-valued
at $0,\ldots,w$. Formula~\eqref{eq:newton} makes it integer-valued at
every integer and gives \eqref{eq:rounding-factorial}. The equation
defining $C_{ij}$ is precisely the $(i,j)$ entry of $AC=Q$.
Induction through all gaps and blocks completes the construction.
\end{proof}

The identity in \eqref{eq:rounding} holds on the whole real line, and
$C(n)$ belongs to $\Gamma$ at every integer. Only the first $M$ samples
of $Q$, however, are known to be canonical. Proposition~\ref{prop:matrix-obstruction} will control the remaining samples by a different argument.

\subsection{A continuous size function}
Define
\begin{equation}\label{eq:shell}
 F(A)=\min_{\gamma\in\Gamma}\sum_{i<j}(A\gamma)_{ij}^{\,2},
 \qquad A\in G.
\end{equation}
The sum ranges over all off-diagonal entries in every block.

\begin{lemma}\label{lem:shell}
The minimum in \eqref{eq:shell} is attained. The function $F$ is
continuous, satisfies $0\leq F(A)\leq\nu/4$, and is invariant under
right multiplication by $\Gamma$. On every compact set of matrices,
only finitely many elements of $\Gamma$ can occur as minimizers.
\end{lemma}
\begin{proof}
The canonical representative gives a candidate with squared-coordinate
sum at most $\nu/4$. If $A$ ranges over a fixed compact set and
$\sum(A\gamma)_{ij}^2\leq\nu/4$, then $A\gamma$ lies in a bounded
coordinate box. The matrices $\gamma=A^{-1}(A\gamma)$ are therefore
bounded as well. Only finitely many integer matrices have this property.
Thus the minimum is attained, with a common finite set of candidates on
each compact set. Locally $F$ is consequently the minimum of finitely
many continuous functions. Right-lattice invariance is immediate from the definition.
\end{proof}

\begin{lemma}\label{lem:no-flat}
If $Q:\R\to G$ is a nonconstant map with ordinary polynomial entries,
then $F(Q(t))$ is not constant on any nondegenerate real interval.
\end{lemma}
\begin{proof}
It is enough to consider a compact nondegenerate interval. Its image
admits only finitely many minimizing matrices $\gamma_1,\ldots,\gamma_s$.
For each of them,
\[
 f_a(t)=\sum_{i<j}(Q(t)\gamma_a)_{ij}^{\,2}
\]
is a nonconstant polynomial. Indeed, $Q(t)\gamma_a$ is nonconstant
because multiplication by $\gamma_a$ is invertible. If $b\geq1$ is
the maximal degree of its coordinates, then the coefficient of $t^{2b}$
in $f_a$ is the sum of the squares of their degree-$b$ coefficients
and is strictly positive.
If $F(Q(t))$ were a constant $c$, the finitely many equations
$f_a(t)=c$ would cover the interval. Each has only finitely many solutions,
which is impossible.
\end{proof}

\subsection{Uniform oscillation and all-color sampling}
Fix the escape radius $\rho=1/4$.
Let $\Kpol$ be the family of polynomial matrix curves $Q:[0,1]\to G$ that have coordinate degrees at most $D$ and satisfy
\begin{equation}\label{eq:normalized}
 \max_{0\leq t\leq1}\coord{Q(t)}\leq1,
 \qquad \coord{Q(1)-Q(0)}=\rho.
\end{equation}
For this auxiliary family we need only the common degree bound $D$.
The family is nonempty: set one coordinate equal to $\rho t$ and all
others to zero.

Since the coordinate degrees are bounded by \(D\), \(\mathcal K\) lies in a finite-dimensional space of polynomial curves.
The bound \(\max_{0\le t\le1}\coord{Q(t)}\le1\) bounds all coefficients by interpolation, while the defining conditions are closed; hence \(\mathcal K\) is compact.
The function
\[
 Q\longmapsto\osc_{[0,1]}F(Q(t)) := \sup_{x,y \in [0,1]}|F(Q(x))-F(Q(y))|
\]
is continuous on this compact family. Every member is nonconstant, so
Lemma~\ref{lem:no-flat} gives a constant $\eta>0$, depending only on
$G,D$, such that
\begin{equation}\label{eq:uniform-osc}
 \osc_{[0,1]}F(Q(t))\geq\eta\qquad(Q\in\Kpol).
\end{equation}

Fix the desired number $r\geq2$ of primary colors and put
\begin{equation}\label{eq:kappa}
 h=\frac{\eta}{20r},\qquad
 \kappa(A)=\left\lfloor\frac{F(A)}h\right\rfloor\pmod r,
\end{equation}
with values in $\{0,\ldots,r-1\}$. Consecutive shells of width $h$
cycle through all colors. Invariance of $F$ gives
\begin{equation}\label{eq:kappa-invariance}
 \kappa(A\gamma)=\kappa(A)\qquad(\gamma\in\Gamma).
\end{equation}

The following lemma replaces the density argument from
the Bohr case. There, sufficiently fine samples along a line segment
meet the interiors of the shells it crosses. Here we work with
polynomial curves of bounded degree. The preceding oscillation
bound ensures that every curve in $\mathcal K$ crosses shells of all
$r$ colors, and the compactness of $\mathcal K$ ensures that a sufficiently
fine parameter mesh, chosen uniformly over $\mathcal K$,
samples every color.
By sampling near shell midpoints, we avoid boundary issues.

\begin{lemma}[Uniform sampling]\label{lem:sampling}
There exists $T_0\geq1$, depending only on $G,D,r$, such that for every
$Q\in\Kpol$ and every real $T\geq T_0$, the samples
\begin{equation}\label{eq:mesh-samples}
 Q(j/T),\qquad 0\leq j\leq\lfloor T\rfloor,
\end{equation}
receive all $r$ colors under $\kappa$.
\end{lemma}
\begin{proof}
The range of $F\circ Q$ is a closed interval $[a,b]$ of length at least
$\eta=20rh$. With $j_0=\lceil a/h\rceil$, the $r$ midpoints
\[
 z_s=(j_0+s+\tfrac12)h,\qquad0\leq s<r,
\]
all belong to $[a,b]$. Their bands have distinct colors, and each
midpoint is distance $h/2$ from the corresponding band boundaries.
Choose $u_s\in[0,1]$ with $F(Q(u_s))=z_s$.

The map $(Q,t)\mapsto F(Q(t))$ is uniformly continuous on the compact
set $\Kpol\times[0,1]$. There is therefore $\theta>0$, independent of
$Q$, such that $|t-u|<\theta$ implies
$|F(Q(t))-F(Q(u))|<h/4$. Choose $T_0>1/\theta$ and approximate $u_s$
by $\lfloor Tu_s\rfloor/T$. The corresponding sampled value remains
strictly inside the band centered at $z_s$. This samples every color,
also when $T$ is not an integer.
\end{proof}

\subsection{Interpolation and the auxiliary parameters}
We need two elementary polynomial estimates. The first says that
a polynomial of bounded degree which is bounded at sufficiently many
consecutive integers varies only slightly on a fixed initial interval.
This ensures that any excursion of our polynomial lift takes long
enough for the integer samples to detect all colors. The second says
that a polynomial which remains bounded on a long real interval has
a small top-degree coefficient.
\begin{lemma}\label{lem:interpolation}
\begin{enumerate}
\item For fixed $D,T_0$ and $\rho>0$, there is an integer $M\geq D+2$
such that every real polynomial $f$ of degree at most $D$ with
$|f(n)|\leq1/2$ for $0\leq n<M$ satisfies
\begin{equation}\label{eq:initial-motion}
 |f(t)-f(0)|<\rho\qquad(0\leq t\leq T_0).
\end{equation}
\item If $\deg f\leq w$, $w\geq1$, and $|f(t)|\leq B$ on $[0,J]$,
where $J>0$, then
\begin{equation}\label{eq:coefficient-bound}
 |[t^w]f(t)|\leq b_w B J^{-w},\qquad
 b_w=\frac{2^ww^w}{w!}.
\end{equation}
\end{enumerate}
\end{lemma}
\begin{proof}
For \textup{(i)}, let
\[
 \lambda_j(u)
 =
 \prod_{\substack{0\leq m\leq D\\m\neq j}}
 \frac{u-m}{j-m},
 \qquad 0\leq j\leq D,
\]
be the Lagrange basis for the nodes \(0,\ldots,D\), and choose
\[
 B_D>\frac12\sum_{j=0}^D\sup_{0\leq u\leq1}|\lambda_j'(u)|.
\]
Choose an integer
\[
 H>\max\{1,T_0,B_DT_0/\rho\},
 \qquad M=DH+1.
\]
If \(\deg f\leq D\) and \(|f(n)|\leq1/2\) for \(0\leq n<M\), then
Lagrange interpolation applied to \(\widetilde f(u)=f(Hu)\) gives
\[
 \widetilde f(u)=\sum_{j=0}^D f(jH)\lambda_j(u),
 \qquad
 \sup_{0\leq u\leq1}|\widetilde f'(u)|\leq B_D.
\]
Hence, for \(0\leq t\leq T_0\),
\[
 |f(t)-f(0)|
 =
 \left|\widetilde f(t/H)-\widetilde f(0)\right|
 \leq \frac{B_DT_0}{H}<\rho.
\]

For \textup{(ii)}, put \(s=J/w\). By the formula for an iterated
finite difference,
\[
 \Delta_s^w f(0)
 =
 \sum_{j=0}^w(-1)^{w-j}\binom{w}{j}f(js),
\]
and therefore
\[
 |\Delta_s^w f(0)|\leq 2^wB.
\]
On the other hand, \(\Delta_s^w\) annihilates every polynomial of
degree less than \(w\), while
\[
 \Delta_s^w(t^w)=w!s^w.
\]
Consequently,
\[
 \Delta_s^w f(0)=w!s^w[t^w]f(t),
\]
and hence
\[
 |[t^w]f(t)|
 \leq \frac{2^wB}{w!s^w}
 =\frac{2^ww^w}{w!}\,BJ^{-w}
 =b_wBJ^{-w}.
\]
\end{proof}

Choose $M$ from Lemma~\ref{lem:interpolation}(i) for the already fixed
$D,T_0,\rho$. Set $\xi=2^{-D-3}$ and partition the canonical box
$I^\nu$ into $R=2^{(D+3)\nu}$ half-open boxes of side length $\xi$.
The label of the box containing $\rep(A)$ is an \emph{auxiliary}
$R$-coloring; it is distinct from the primary $r$-coloring $\kappa$.
Finite van der Waerden \cite{vdW} supplies $W=W(R,M)$ such that every
$R$-coloring of $\{0,\ldots,W-1\}$ has a monochromatic $M$-term
progression. Finally set
\begin{equation}\label{eq:fixed-L}
 L=D!\,\lcm(1,2,\ldots,W).
\end{equation}
The order of choices is
\begin{equation}\label{eq:choice-order}
 (\text{weight patterns},r)\longrightarrow
 G,\Gamma,D,\eta,h,\kappa,T_0\longrightarrow M,\xi,R,W\longrightarrow L.
\end{equation}
These objects are fixed before any coefficients or any progression are
given. 
The auxiliary coloring only finds a short subprogression; it does
not impose an $R$-largeness hypothesis on the allowed differences.

\subsection{Control of the leading coefficients}
\begin{proposition}\label{prop:matrix-obstruction}
Let $p\in\Hpol$ and $k\geq4W+2$. If
$\{\kappa(p(j)):0\leq j<k\}$ omits a color, there are
$\gamma\in\Gamma$ and $E\in G$ such that
\begin{equation}\label{eq:leading-matrix}
 \delta_L\tau(p)\,\gamma=E,
 \qquad
 |E_{ij}|\leq\frac34b_{w_{ij}}(2L)^{w_{ij}}k^{-w_{ij}}.
\end{equation}
\end{proposition}
\begin{proof}
Let $e$ be a color omitted by these samples. Apply the auxiliary grid
coloring to $p(0),\ldots,p(W-1)$. The choice of $W$ gives $b\geq0$
and $q\geq1$ with
\[
 b+(M-1)q\leq W-1
\]
such that the canonical representatives of
$p(b),p(b+q),\ldots,p(b+(M-1)q)$ lie in one grid box. In particular,
$b,q\leq W$. The rounding lemma, applied to $p(b+qt)$, gives
\begin{equation}\label{eq:polynomial-lift}
 p(b+qt)C(t)=Q(t),\qquad C(\Z)\subseteq\Gamma,
\end{equation}
where $Q,C\in\Hpol$ and $Q(0),\ldots,Q(M-1)$ are canonical.
Put
\[
 J=\left\lfloor\frac{k-1-b}{q}\right\rfloor\geq M-1.
\]
For all integers $0\leq j\leq J$, right-lattice invariance gives
\begin{equation}\label{eq:omitted-lift}
 \kappa(Q(j))=\kappa(p(b+qj))\neq e.
\end{equation}

We claim that
\begin{equation}\label{eq:no-escape}
 \coord{Q(t)-Q(0)}<\rho\qquad(0\leq t\leq J).
\end{equation}
The first $M$ samples have all coordinates bounded by $1/2$.
Lemma~\ref{lem:interpolation}(i) therefore gives this displacement bound
up to time $T_0$. If \eqref{eq:no-escape} failed, continuity would give
a first real time $T\leq J$ when the displacement equals $\rho$.
Necessarily $T>T_0$. Up to that time the coordinates of $Q$ are bounded
by $1/2+\rho=3/4$, since $Q(0)$ is canonical. The rescaled curve
$\widetilde Q(u)=Q(Tu)$ belongs to $\Kpol$.
Lemma~\ref{lem:sampling}, with mesh parameter $T$, says that the samples
\[
 \widetilde Q(j/T)=Q(j),\qquad0\leq j\leq\lfloor T\rfloor,
\]
receive all $r$ colors.
These samples occur in \eqref{eq:omitted-lift}, a contradiction.
This proves \eqref{eq:no-escape}.

Every coordinate of $Q$ is now bounded by $3/4$ on the entire real
interval $[0,J]$. Its individual degree bound and
Lemma~\ref{lem:interpolation}(ii) give
\begin{equation}\label{eq:tau-Q-small}
 |\tau(Q)_{ij}|\leq\frac34b_{w_{ij}}J^{-w_{ij}}.
\end{equation}
Applying $\tau$ to \eqref{eq:polynomial-lift} yields the exact identity
\begin{equation}\label{eq:extracted-lift}
 (\delta_q\tau(p))\tau(C)=\tau(Q).
\end{equation}
The matrix $\tau(C)$ need not belong to $\Gamma$, even though $C(n)$
does for every integer $n$. We clear its factorial denominators at the
same time as we remove the auxiliary step $q$.
Apply $\delta_{L/q}$ to \eqref{eq:extracted-lift} and define
\[
 \gamma=\delta_{L/q}\tau(C),\qquad E=\delta_{L/q}\tau(Q).
\]
By \eqref{eq:fixed-L}, $L/q$ is an integer multiple of $D!$.
Writing $\tau(C)_{ij}=z_{ij}/w_{ij}!$ with $z_{ij}\in\Z$, as permitted
by \eqref{eq:rounding-factorial}, gives
\begin{equation}\label{eq:clearance}
 \gamma_{ij}=\frac{(L/q)^{w_{ij}}}{w_{ij}!}\,z_{ij}\in\Z.
\end{equation}
Thus $\gamma\in\Gamma$ and $\delta_L\tau(p)\gamma=E$.
Moreover, \eqref{eq:tau-Q-small} yields
\[
 |E_{ij}|\leq\frac34 b_{w_{ij}}\left(\frac{L}{qJ}\right)^{w_{ij}}.
\]
Finally,
$qJ\geq k-1-b-q\geq k-1-2W\geq k/2$, which proves
\eqref{eq:leading-matrix}.
\end{proof}

Only the short initial subprogression lies in one auxiliary grid box.
Boundedness on the long real interval is a consequence of the primary
omitted-color hypothesis. The use of the actual real first-exit time
is essential: an exit need not occur at an integer.

\subsection{Recovering the special generalized polynomials} \label{ssec:recovery}

For real $x_1,\ldots,x_\ell$, let $B(x_1,\ldots,x_\ell)$ be the
upper-unitriangular matrix whose first superdiagonal is
\begin{equation}\label{eq:band-signs}
 -x_1,-x_2,\ldots,-x_{\ell-1},+x_\ell,
\end{equation}
and whose higher off-diagonal entries are zero. For $\ell=1$ its
only off-diagonal entry is $+x_1$.

\begin{lemma}[Last-column decoding]\label{lem:decode}
Suppose
$B(x_1,\ldots,x_\ell)\gamma=E$, where
$\gamma\in\UT_{\ell+1}(\Z)$ and
$E_{i,\ell+1} \in [-\frac{1}{2},\frac{1}{2})$ for $1\leq i\leq\ell$. Then
\begin{equation}\label{eq:decode}
 \mathsf L(x_1,\ldots,x_\ell)+\gamma_{1,\ell+1}=E_{1,\ell+1},
 \qquad
 \tnorm{\mathsf L(x_1,\ldots,x_\ell)}=|E_{1,\ell+1}|.
\end{equation}
\end{lemma}
\begin{proof}
Put $H_\ell=x_\ell$ and $H_i=x_i\nint{H_{i+1}}$ downward, so
$H_1=\mathsf L(x_1,\ldots,x_\ell)$. The last-column equation in row
$\ell$ is
\[
 E_{\ell,\ell+1}=x_\ell+\gamma_{\ell,\ell+1}.
\]
Since $E_{\ell,\ell+1} \in [-\frac{1}{2},\frac{1}{2})$, we have
$\gamma_{\ell,\ell+1}=-\nint{H_\ell}$ uniquely. For $i<\ell$, the
band structure gives
\[
 E_{i,\ell+1}=\gamma_{i,\ell+1}-x_i\gamma_{i+1,\ell+1}.
\]
If $\gamma_{i+1,\ell+1}=-\nint{H_{i+1}}$, this becomes
$E_{i,\ell+1}=\gamma_{i,\ell+1}+H_i$. Smallness again identifies the
integer correction exactly: $\gamma_{i,\ell+1}=-\nint{H_i}$.
Induction proves the top-row identity. Since its residual lies in
$[-1/2,1/2)$, distance to the integers is its absolute value.
\end{proof}

\begin{example}\label{ex:Reduction}
Let $P(n)=\alpha n^u\nint{\beta n^v\nint{\gamma n^w}}$ and observe that
\begingroup
\setlength{\arraycolsep}{3pt}
\renewcommand{\arraystretch}{1.25}
\[
\begin{aligned}
&
\begin{pmatrix}
 1 & -\alpha n^u & 0 & 0\\
 0 & 1 & -\beta n^v & 0\\
 0 & 0 & 1 & \gamma n^w\\
 0 & 0 & 0 & 1
\end{pmatrix}
\begin{pmatrix}
 1
 & -\nint{-\alpha n^u}
 & -\nint{\alpha n^u\nint{-\beta n^v}}
 & -\nint{\alpha n^u\nint{\beta n^v\nint{\gamma n^w}}}\\
 0
 & 1
 & -\nint{-\beta n^v}
 & -\nint{\beta n^v\nint{\gamma n^w}}\\
 0 & 0 & 1 & -\nint{\gamma n^w}\\
 0 & 0 & 0 & 1
\end{pmatrix}
\\[8pt]
&\quad =
\begin{pmatrix}
 1
 & \cpart{-\alpha n^u}
 & \cpart{\alpha n^u\nint{-\beta n^v}}
 & \cpart{\alpha n^u\nint{\beta n^v\nint{\gamma n^w}}}\\
 0
 & 1
 & \cpart{-\beta n^v}
 & \cpart{\beta n^v\nint{\gamma n^w}}\\
 0 & 0 & 1 & \cpart{\gamma n^w}\\
 0 & 0 & 0 & 1
\end{pmatrix}.
\end{aligned}
\]
\endgroup
The second matrix on the left belongs to \(\UT_4(\Z)\), and every
entry above the diagonal of the matrix on the right belongs to
\([-1/2,1/2)\). Thus the latter is the canonical representative
\(\rep(B(\alpha n^u,\beta n^v,\gamma n^w))\).
Its top-right entry is \(\cpart{P(n)}\), whose absolute value is
\(\tnorm{P(n)}\).
\end{example} 
\subsection{The proof}\label{ssec:TheProof}

\begin{proof}[Proof of Theorem~\ref{thm:omitted}]
First fix all the weight patterns and carry out the choices in
\eqref{eq:choice-order}. For one expression
$P(n)=\mathsf L(\alpha_1n^{w_1},\ldots,\alpha_\ell n^{w_\ell})$, set
\[
 g_P(t)=B(\alpha_1t^{w_1},\ldots,\alpha_\ell t^{w_\ell}).
\]
Let $g(t)\in G$ be the tuple of these matrices for all $m$ expressions.
Every nonzero off-diagonal coordinate of $g$ is homogeneous of its
assigned weight, and $g\in\Hpol$. Define
\begin{equation}\label{eq:integer-color}
 c(n)=\kappa\bigl(g(n/L)\bigr),\qquad n\in\N.
\end{equation}
The coefficients enter only through $g$; the coloring $\kappa$ and
integer $L$ depend on the patterns and $r$ alone.

Suppose $\AP{a}{d}{k}$ omits a color, and put
\[
 p(t)=g\bigl((a+dt)/L\bigr).
\]
The first $k$ integer samples of this member of $\Hpol$ omit that same
color.
Entrywise coefficient extraction, using \eqref{eq:scalar-extraction}, gives \begin{equation}\label{eq:tau-g} \tau(p)=g(d/L),\qquad \delta_L\tau(p)=g(d).
\end{equation}
Indeed, a superdiagonal entry $\pm\alpha_i((a+dt)/L)^{w_i}$ contributes $\pm\alpha_i(d/L)^{w_i}$, and all higher entries remain zero.

Proposition~\ref{prop:matrix-obstruction} therefore gives
\begin{equation}\label{eq:g-small}
 g(d)\gamma=E,\qquad \gamma\in\Gamma,\qquad
 |E_{ij}|\leq C_*k^{-w_{ij}},
\end{equation}
where
\[
 C_* =\max\left\{1,\max_{1\leq w\leq D}\frac34 b_w(2L)^w\right\}.
\]
For $k\geq\max\{4W+2,\lceil4C_*\rceil+1\}$ all off-diagonal errors
are less than $1/4$. Apply Lemma~\ref{lem:decode} in each block.
The top-right coordinate in the block for $P_i$ has weight $d_i$;
hence
\[
 \tnorm{P_i(d)}=|E_{\mathrm{top},i}|\leq C_*k^{-d_i}.
\]

Finally, for each $d\leq D$ there are only finitely many ordered lists
of positive integers with sum $d$. There are therefore finitely many
$m$-tuples of admissible weight patterns. Taking the maxima of the
constants and thresholds over these tuples gives $C(r,m,D)$ and
$k_0(r,m,D)$. For a given coefficient family, use the coloring
\eqref{eq:integer-color} for its pattern tuple.
\end{proof}
\section{Nilmanifolds and a 2-large set that is not 3-large}\label{sec:Nilmanifolds} 

\subsection{Nilsystems and the question of Host, Kra, and Maass}\label{ssec:Nilrotations}
A topological dynamical system is a pair $(X,T)$ in which $X$ is a compact metric space and $T:X\rightarrow X$ is a continuous map. In this section we will only be concerned with invertible topological dynamical systems, so we will have the standing assumption that $T$ is a homeomorphism.
The system $(X,T)$ is \emph{minimal} if for every $x \in X$, the set $\{T^nx\}_{n = 1}^\infty$ is dense in $X$.
Given a point $x \in X$ and a nonempty open set $U \subseteq X$, we define the pointwise return times $N(x,U) := \{n \in \mathbb{N}\ |\ T^nx \in U\}$.
Given an $\ell \in \mathbb{N}$ and a nonempty open set $U \subseteq X$, we define the $\ell$-return times $N^\ell(U) := \{n \in \mathbb{N}\ |\ U\cap T^nU\cap T^{2n}U\cap\cdots\cap T^{\ell n }U \neq \emptyset\}$.

An $s$-step \emph{nilmanifold} is a quotient space \(X=G/\Gamma\), where \(G\) is a finite-dimensional $s$-step nilpotent Lie group and \(\Gamma\) is a discrete cocompact subgroup of \(G\).
Standard examples of nilmanifolds, as we have seen earlier in this paper, are the quotients \(\UT_n(\R)/\UT_n(\Z)\) and their finite products.
Given a $g \in G$, the map $T_g:X\rightarrow X$ given by $T_g(x\Gamma) = gx\Gamma$ is a \emph{nilrotation}, and the pair $(X,T_g)$ is a \emph{nilsystem}.
The nilsystem $(X,T_g)$ is minimal if and only if there exists $x \in X$ for which $\{T_g^nx\}_{n = 1}^\infty$ is dense in $X$.

Bergelson and Leibman \cite{BergelsonLeibman,LeibmanCanonical} showed that generalized polynomials and nilrotations are naturally related to each other.
The following manifestation of this relationship is a consequence of Theorems B and 4.2.11 of Huang, Shao, and Ye \cite{HSY}.

\begin{theorem}\label{thm:HSYThmA}
    If $(X,T_g)$ is a minimal nilsystem, $x \in X$, and $U$ is an open neighborhood of $x$, then $N(x,U)$ contains a nil-Bohr neighborhood\textsuperscript{\ref{fn:first}} of $0$.
\end{theorem}

Host, Kra, and Maass \cite{HKM} observed that if $(X,T_g)$ is a minimal nilsystem, $x \in X$, and $U$ is a nonempty open set for which $x \notin \overline{U}$, then $N(x,U)$ is not large.
They then ask in Question 7.8 whether there exists $r \ge 2$ for which $N(x,U)$ fails to be $r$-large.
Theorem \ref{thm:HSYThmA} tells us that $N(x,U)$ is not a set of nil-Bohr recurrence, so Theorem \ref{thm:2LargeIsNilrecurrent}(i) tells us that $N(x,U)$ is not $2$-large, regardless of the choice of $U$ and $x \notin \overline{U}$.

\subsection{Factor maps and \texorpdfstring{$\infty$}{infinity}-step pro-nilsystems}\label{ssec:InftyStepProNilSystems}
Recall that the system $(Y,S)$ is a \emph{factor} of $(X,T)$ if there exists a continuous surjection $\pi:X\rightarrow Y$, called a \emph{factor map}, satisfying $\pi\circ T = S\circ\pi$. 
Now suppose that $\{(X_i,T_i)\}_{i = 1}^\infty$ is a family of systems, and we are also given factor maps $\pi_i:X_{i+1}\rightarrow X_i$.
The \emph{inverse limit} of $\{(X_i,T_i)\}_{i = 1}^\infty$ is the system $(X,T)$ given by $X = \{(x_i)_{i = 1}^\infty \in \prod_{i = 1}^\infty X_i\ |\ \pi_ix_{i+1} = x_i\ \forall\ i \in \mathbb{N}\}$ and $T(x_i)_{i = 1}^\infty = (T_ix_i)_{i = 1}^\infty$, where $\prod_{i = 1}^\infty X_i$ is given the product topology and $X$ is given the subspace topology.
An \emph{$\infty$-step pro-nilsystem} is an inverse limit of a family of nilsystems.
Our next result follows easily from \cite[Corollary 5.3]{HKM}, but we still give a proof for the sake of completeness.

\begin{proposition}
    For $S \subseteq \mathbb{N}$ the following are equivalent:
    \begin{enumerate}
        \item $S$ is a set of nil-Bohr recurrence.

        \item $S$ is a set of pointwise recurrence for $\infty$-step pro-nilsystems.
        In other words, for every minimal $\infty$-step pro-nilsystem $(X,T)$, every $x \in X$, and every open neighborhood $U$ of $x$, we have $S\cap N(x,U) \neq \emptyset$.

        \item $S$ is a set of multiple topological recurrence for $\infty$-step pro-nilsystems.
        In other words, for every minimal $\infty$-step pro-nilsystem $(X,T)$, every nonempty open set $U \subseteq X$, and every $\ell \in \mathbb{N}$, we have $S\cap N^\ell(U) \neq \emptyset$.
    \end{enumerate}
\end{proposition}

\begin{proof}
    Theorems B and 4.2.11 of \cite{HSY} tell us that $S$ is a set of nil-Bohr recurrence if and only if it is a set of pointwise recurrence for all minimal nilsystems.
    In other words, $S$ is a set of nil-Bohr recurrence if and only if for every minimal nilsystem $(X,T)$, every $x \in X$, and every open neighborhood $U$ of $x$, there exists $s \in S$ for which $T^sx \in U$.
    We recall from \cite[Corollary 5.3]{HKM} that $S$ is a set of multiple topological recurrence for all minimal nilsystems if and only if it is a set of pointwise recurrence for all minimal nilsystems.
    Consequently, we see that (i) is equivalent to (ii) and (iii) restricted to the class of minimal nilsystems.
    Since every nilsystem is an $\infty$-step pro-nilsystem, we see that (ii) and (iii) both imply (i).

    It remains to show that (i) implies (ii) and (iii).
    Let $(X,T)$ be a minimal $\infty$-step pro-nilsystem, let $U \subseteq X$ be a nonempty basic open set, let $x \in U$, and let $\ell \in \mathbb{N}$ be arbitrary.
    Let $n \ge \ell$ and $U' = U_1\times U_2\times\cdots\times U_n\times \prod_{i = n+1}^{\infty}X_i$ be such that $U = U'\cap X$.
    Let $p_n:X\rightarrow X_n$ denote the natural factor map given by projection onto the $n^{\text{th}}$ coordinate and for $1 \le i \le n-1$ let $\sigma_i:X_n\rightarrow X_i$ be the factor map given by $\sigma_i = \pi_i\circ\cdots\circ\pi_{n-1}$.
    Observe that $p_n(U) = U_n\cap\bigcap_{i = 1}^{n-1}\sigma_i^{-1}U_i$ is an open set in $X_n$.
    Since $X_n$ is a nilsystem, we have $S\cap N(p_n(x),p_n(U)) \neq \emptyset$ and $S\cap N^\ell(p_n(U)) \neq \emptyset$.
    To finish, it suffices to observe that $N(p_n(x),p_n(U)) \subseteq N(x,U)$ and $N^\ell(p_n(U)) \subseteq N^\ell(U)$.
\end{proof}

Given an arbitrary minimal system $(X,T)$, the \textbf{maximal $\infty$-step pro-nilfactor} is a factor denoted by $(X_\infty,T)$ that is an $\infty$-step pro-nilsystem.
Furthermore, $(X_\infty,T)$ has the property that if $(Y,S)$ is any other $\infty$-step pro-nilsystem that is a factor of $(X,T)$, then $(Y,S)$ is also a factor of $(X_\infty,T)$.
It is now well-known that nilsystems and the maximal $\infty$-step pro-nilfactor play a crucial role in the theory of multiple recurrence \cite{HostKraBook,HostKraMaassStructureTheorem,HostKraNonconventialAverages,ZieglerUniversalCharacteristicFactor,GlasnerEtAlCharacteristicFactor,QiuDenseOrbits,StructureTheoremForPolynomialReturnTimesInMinimalSystems,FrantzikinakisMulticorrelation,HSY}.
To state our next result, we recall that if $C$ is a subset of a topological space $X$, then $C^\circ$ denotes the interior of $C$, which is the union of all open subsets of $C$.

\begin{theorem}[{cf. \cite[Theorem A]{StructureTheoremForPolynomialReturnTimesInMinimalSystems}}]\label{thm:StructureTheoremForPolynomialOrbits}
    Let $(X,T)$ be a minimal system with maximal $\infty$-step pro-nilfactor $(X_\infty,T)$ and associated factor map $\pi:X\rightarrow X_\infty$.
    For all nonempty open $U \subseteq X$ and all $\ell \in \mathbb{N}$, the set $N^\ell(\pi(U)^\circ)\setminus N^\ell(U)$ is not piecewise syndetic.    
\end{theorem}
\subsection{A piecewise syndetic set of nil-Bohr recurrence is \texorpdfstring{$2$}{2}-large}\label{ssec:WhennilBohrIs2Large}

The general bridge between partition Ramsey theory and minimal topological dynamics comes from the Furstenberg correspondence principle \cite{FurstenbergBook,FurstenbergWeissCorrespondence}.
For our purposes, we only record the following special case of this principle that is a combination of \cite[Proposition 7.2]{HKM} and \cite[Lemma 1]{PropertiesOfAccessibleSets}.

\begin{proposition}\label{prop:Dynamically2Large}
    Let $r \ge 2$.
    The set $S \subseteq \mathbb{N}$ is $r$-large if and only if for every minimal system $(X,T)$, every cover $\mathcal{U} = (U_1,\cdots,U_r)$ of $X$ by $r$ clopen sets, and every $\ell \ge 1$, there exists $j \in \{1,\cdots,r\}$ such that $S\cap N^\ell(U_j) \neq \emptyset$.
\end{proposition}

We are now ready to prove that if $S$ is piecewise syndetic and a set of nil-Bohr recurrence, then $S$ is $2$-large.

\begin{proof}[Proof of Theorem~\ref{thm:2LargeIsNilrecurrent}(ii)]
Suppose, toward a contradiction, that $S$ is not $2$-large.
Proposition \ref{prop:Dynamically2Large} provides us with a minimal system $(X,T)$, a clopen cover $(U_1,U_2)$ of $X$, and an $\ell \in \mathbb{N}$ for which $S\cap N^\ell(U_1) = S\cap N^\ell(U_2) = \emptyset$.
Let $\pi:X\to X_\infty$ denote the factor map onto the maximal $\infty$-step pro-nilfactor of $(X,T)$.

First suppose that $\pi(U_1)=X_\infty$.
Then $\pi(U_1)^\circ=X_\infty$, and hence
\begin{equation}
N^\ell(\pi(U_1)^\circ)\setminus N^\ell(U_1)
=
\N\setminus N^\ell(U_1).
\end{equation}
Theorem \ref{thm:StructureTheoremForPolynomialOrbits} tells us that the set on the left is not piecewise syndetic, which contradicts the fact that it contains $S$.

Consequently, $\pi(U_1)$ is a proper subset of $X_\infty$.
Since $U_1$ is compact and $\pi$ is continuous, $\pi(U_1)$ is closed.
Therefore
\begin{equation}
V=X_\infty\setminus\pi(U_1)
\end{equation}
is a nonempty open subset of $X_\infty$.
Moreover, since $U_1$ and $U_2$ cover $X$, we have $\pi^{-1}(V)\subseteq U_2$, hence $N^\ell(V) \subseteq N^\ell(U_2)$.
Since $(X_\infty,T)$ is a minimal $\infty$-step pro-nilsystem, Proposition 4.2 tells us that $S\cap N^\ell(V) \neq \emptyset$, which contradicts our assumption that $S\cap N^\ell(U_2) = \emptyset$.
\end{proof}

\begin{corollary}\label{cor:2ButNot3Large}
    There exists a set $S$ that is $2$-large but not $3$-large.
\end{corollary}

\begin{proof}
    In \cite{AlweissObstacles} the first author constructed a set of nil-Bohr recurrence $E$ that is not $2$-large.
    Theorem \ref{thm:2LargeIsNilrecurrent}(ii) shows us that $S := (\mathbb{N}\setminus3\mathbb{N})\cup E$ is $2$-large.

    Now let $\ell \in \mathbb{N}$ and $\chi:\mathbb{N}\rightarrow\{0,1\}$ be such that $\chi$ admits no monochromatic $\ell$-term arithmetic progression with common difference in $E$.
    Let $\chi':\mathbb{N}\rightarrow\{0,1,2\}$ be given by $\chi'(n) = n+\chi(n)\pmod{3}$.
    We see that for $s \in \mathbb{N}\setminus3\mathbb{N}$ and $a \in \mathbb{N}$, the triple $\{a,a+s,a+2s\}$ represents every residue class modulo $3$ and simple casework shows it cannot be monochromatic under $\chi'$.
    For $s \in E \cap 3\N$ and $a,i,j \in \mathbb{N}$, we see that $\chi'(a+is) = \chi'(a+js)$ if and only if $\chi(a+is) = \chi(a+js)$, hence $\chi'$ admits no monochromatic $\ell$-term arithmetic progressions with common difference in $E \cap 3\N$.
\end{proof}

\begin{remark}
    We recall that the set $\mathbb{N}\setminus3\mathbb{N}$ is not $2$-large, so we have also shown in Corollary \ref{cor:2ButNot3Large} that $2$-large sets are not partition regular.
    In fact, the proof of Corollary \ref{cor:2ButNot3Large} can be modified to show that if $E$ is not 2-large, then $S := E\cup(\mathbb{N}\setminus2\mathbb{N})$ is not $3$-large.
    Consequently, if we further assume that $E$ is a set of nil-Bohr recurrence, then $S$ will be $2$-large but not $3$-large.
\end{remark}

\section{Application to Szemer\'edi numbers}\label{sec:extremal}

We now prove Corollary \ref{cor:extremal} concerning Szemer\'edi numbers.
We first require the following result, which bootstraps the constant from our main theorem by applying it to an appropriate auxiliary collection of generalized polynomials.

\begin{lemma}\label{lem:constant}
In Theorem~\ref{thm:omitted}, the constant $C = C(r,m,D)$ may be replaced by any prescribed $c > 0$, at the cost of increasing the threshold from $k_0(r,m,D)$ to $k_1 = k_1(r,m,D,c)$.
\end{lemma}
\begin{proof}
Let $C_0 = C(r,2m,D)$ and $k' = k_0(r,2m,D)$ be as in Theorem~\ref{thm:omitted}.
Choose the fixed integer $q=\lceil C_0/c\rceil$.
Apply Theorem \ref{thm:omitted} to $P_1,\ldots,P_m,qP_1,\ldots,qP_m$; each $qP_i$ is a special generalized polynomial of the same degree, since only its outermost coefficient changes.

On a progression omitting a color, for $k\geq k'$, both
$\tnorm{P_i(d)}$ and $\tnorm{qP_i(d)}$ are at most $C_0k^{-d_i}$.
Write $P_i(d)=z_i+u_i$ with $z_i\in\Z$ and
$|u_i|=\tnorm{P_i(d)}$.
If $k\geq4qC_0$, then $|qu_i|\leq1/4$,
so there is no wrap around the circle and
\[
 q|u_i|=\tnorm{qP_i(d)}\leq C_0k^{-d_i}.
\]
Thus $\tnorm{P_i(d)}\leq(C_0/q)k^{-d_i}\leq ck^{-d_i}$.
Both $q$ and the coloring were fixed independently of $k$, so it suffices to take $k_1(r,m,D,c) = \max\{\lceil 4qC_0 \rceil,k'\}$.
\end{proof}

With this lemma in hand, we can straightforwardly apply it and delete the least common color class to pass from a large restricted $k$-AP-free set to a large $k$-AP-free set.

\begin{proof}[Proof of Corollary~\ref{cor:extremal}]
Choose $r\geq2$ such that $r/(r-1)\leq1+\eps$, and take the coloring
from Lemma~\ref{lem:constant} for the prescribed $c$.
Fix $k$ above its threshold. Let $A\subseteq\Omega$ contain no
$k$-term progression with difference in $B_k(c)$. Delete a color class
of smallest cardinality in $A$, leaving a set $A'$ with
\[
 |A'|\geq\left(1-\frac1r\right)|A|.
\]
Every progression in $A'$ omits the deleted color, so a $k$-term
progression there would have difference in $B_k(c)$, a contradiction.
Thus $A'$ is $k$-AP-free and
\[
 |A|\leq\frac r{r-1}r_k(\Omega)\leq(1+\eps)r_k(\Omega).
\]
Taking the maximum over $A$ proves the upper bound. The lower bound
holds because forbidding every $k$-term progression is a stronger
restriction than forbidding just the ones with difference in $B_k(c)$.
\end{proof}
\section{Further questions}\label{sec:questions}

Currently, GPT-6 claims to be able to construct for each $k \ge 2$ a $k$-large set that is not $(k+1)$-large.
We have provided some AI-generated claimed proofs of various strengthenings of this claim \cite{Farhangi2026GitHub}, but they have not yet been reviewed by humans.
While  we believe that these results are interesting and worth human review and understanding, we would also like to spur further research by posing some questions and conjectures that seem to currently be out of reach of AI.
The repository \cite{Farhangi2026GitHub} contains the partial progress that GPT-6 obtained on these questions and conjectures.

\begin{conjecture}\label{conj:EveryLargeSetContainsEveryCounterexample}
    Suppose that $S \subseteq \mathbb{N}$ is large.
    For every $k \ge 2$, there exists a set $S_k \subseteq S$ that is $k$-large but not $(k+1)$-large.
\end{conjecture}

In the density setting, however, we cannot establish a similar separation, though we expect it to be true.

\begin{conjecture}\label{conj:SzemerediLargeSeparation}
    For every $\delta \in (0,1)$, there exists a set $S_\delta \subseteq \mathbb{N}$ that is $\delta$-Szemer\'edi large but fails to be $(\delta-\varepsilon)$-Szemer\'edi large for every $\varepsilon \in (0,\delta)$.
\end{conjecture}

Unlike large sets, $2$-large sets are not partition regular. But we do not know if any non-large set can be partitioned into finitely many non-$2$-large pieces.

\begin{question}\label{q:PartitionRegular2LargeIsLarge}
    Suppose that $S \subseteq \mathbb{N}$ is such that for every finite partition $S = \bigcup_{i = 1}^rS_i$, some $S_i$ is $2$-large.
    Must $S$ be large?
\end{question}

In general, it would be interesting to find more interesting examples of non-large sets. The only ones known are sets which are not sets of nil-Bohr recurrence on the one hand, and modifications of the example of \cite{AlweissObstacles} on the other.

\bibliographystyle{amsplain}
\bibliography{vdw_largeness_nilrecurrence}
\end{document}